\documentclass[12pt,reqno]{amsart}

\newtheorem{theorem}{Theorem}[section]
\newtheorem{corollary}[theorem]{Corollary}

\newtheorem{lemma}[theorem]{Lemma}
\newtheorem{proposition}[theorem]{Proposition}

\numberwithin{equation}{section}

\usepackage{amssymb}
\usepackage{mathrsfs}
\usepackage{amsmath}
\usepackage{enumitem}
\usepackage{amsfonts}
\usepackage{colortbl,dcolumn}
\usepackage{color, soul}
\usepackage{amsthm}

\allowdisplaybreaks \numberwithin{equation}{section}

\author{Jiayu Li, Xiangrong Zhu}

\address{Jiayu Li, School of Mathematical Sciences, University of Science and Technology of China Hefei 230026 \\ AMSS CAS Beijing 100190, P. R. China}
\email{jiayuli@ustc.edu.cn}

\address{Xiangrong Zhu (the corresponding author), School of Mathematical Sciences, Zhejiang Normal University, Jinhua 321004, P. R. of China.}
\email{zxr@zjnu.cn}

\keywords{Linear Poisson equation, Zygmund class, Isoperimetric inequality, Harnack's inequality, Hardy space }

\thanks {The first author is supported by National Key R$\&$D Program of China 2022YFA1005400 and NFSC No.12031017. The second author is supported by the National Key R$\&$D Program of China 2023YFA1010800.}

\begin{document}

\title[Linear Poisson equations with potential on Riemann surfaces]
{Linear Poisson Equations with Potential on Riemann Surfaces}

\begin{abstract}
We study interior estimates for solutions of the linear Poisson equation:
$$
\triangle u = g u + f
$$
where $g$ and $f$ belong to the Zygmund space $L\ln L$ on a Riemann surface $M$ satisfying the isoperimetric inequality.
As applications, we derive corresponding interior estimates, Harnack inequalities, and a global estimate.
\end{abstract}

\maketitle

{\bf Mathematics Subject Classification (2020):} 35J15 (primary), 58J10 (secondary).

\section{Introduction}

 Let $\Omega $ be a bounded domain in $\mathbf{R}^2$. Using Moser's iteration, one can derive the Harnack inequality for a linear uniformly elliptic equation
 $$\operatorname{div} (A\nabla u)=gu+f~~{\rm in}~~\Omega$$
 in the case that $g,f\in L^q(\Omega)$ for some $q>1$. Here, no regularity assumption is made on the coefficients. For details, see \cite{M1} or \cite{Trudinger}.

 Let $M$ be a Riemann surface. We say that $M$ satisfies the isoperimetric inequality if there exists $A>0$ such that
 \begin{align}
 V(\Omega)\leq A l(\partial\Omega)^2\label{sla1.2}
 \end{align}
 for any open domain $\Omega\subset M$ with rectifiable boundary.

The isoperimetric inequality implies the Sobolev inequality (see \cite{S-Y}):
\begin{align}
  \|u\|_{L^q(\Omega)}\leq \frac{q\sqrt{A}}{2}V(\Omega)^{\frac{1}{q}}\|\nabla u\|_{L^2(\Omega)} \label{sobolev1.2}
\end{align}
for all $u \in H^1_0(\Omega)$, $q\geq 1$ and all smooth domains $\Omega \subset M$.

 One can check that the Harnack inequality in \cite{M1} can be generalized to a Riemann surface satisfying the isoperimetric inequality.

 It is natural to ask whether we can apply Moser's iteration to $g,f$ in the Hardy space $\mathscr{H}^1$ or the Zygmund class $L\ln L$ to derive the corresponding Harnack inequality.
 To the best of our knowledge, we have not found any result in this direction.

 Take a smooth function $\eta\in C^{\infty}_{c}(B_2)$ with $0\leq \eta\leq 1$ and $\eta\equiv 1$ on $B_1$.
 For $k>10$, set $y_k=(\frac 4k,\frac 4k)$ and
 \begin{align*}
 f_k(x)=k^2\eta(k(x-y_k))-k^2\eta(k(x+y_k)).
 \end{align*}
  By the theory of $\mathscr{H}^1$-atoms (see \cite{CW2}), we know that $\frac{f_k}{36\pi}$ is a standard $\mathscr{H}^1$-atom in $B_{\frac 6k}$ and then
   $$\|f_k\|_{\mathscr{H}^1(\mathbb{R}^2)}\leq 36\pi.$$
  Set $A_k(x)=1+\eta(\frac{k(x+y_k)}{2})$ and let $u_k$ be the solution to $A_k\triangle u_k(x)=f_k(x)$ with $u_k(\infty)=0$.
  Then from the choice of $\eta$, it is easy to see that
  $$\triangle u_k(x)=\frac{f_k(x)}{A_k(x)}=k^2\eta(k(x-y_k))-\frac{k^2}{2}\eta(k(x+y_k)),u_k(\infty)=0.$$
  Now, straightforward computations yield
 \begin{align*}
 |u_k(0)|
 =&\left|\int_{\mathbb{R}^2}\ln |y|[k^2\eta(k(y-y_k))-\frac{k^2}{2}\eta(k(y+y_k))]dy\right|\\
 =&\frac{k^2}{2}\int_{\mathbb{R}^2}\ln \frac{1}{|y|}\eta(k(y-y_k))dy\\
 \geq &\frac{k^2}{2}\int_{|y-y_k|<\frac 1k}\ln \frac{1}{|y|}dy\\
 \geq &\frac{\pi}{2}\ln \frac{k}{5}.
 \end{align*}
 Therefore, for a general linear uniformly elliptic equation with the nonhomogeneous term $f$ in the Hardy space, we cannot expect to obtain an interior estimate if no regularity is assumed on the coefficients.
 However, it is known (c.f. \cite{Helein}) that if $u$ is a solution of $\triangle u=f$ with $f$ in $\mathscr{H}^1$, then $u$ is continuous and uniformly bounded.

 Here we consider the following Poisson equation with potential
\begin{equation}\label{basic eqn}
 \triangle u=gu+f
 \end{equation}
 on a Riemann surface $M$.

 Before presenting the main theorems, we recall some well-known function spaces.

 When the measure of $X$ is finite, the Zygmund class $L\ln L(X)$ consists of all functions $f$ for which
 $$\|f\|_{L\ln L(X)}=\int_{X} |f(x)|\max\{0,\ln |f(x)|\}dx<\infty.$$
 $L\ln L(X)$ is a natural generalization of $L^q(X)$ for $q>1$.

 Note that $\|\cdot\|_{L\ln L(X)}$ is not a norm. So, we need a suitable norm in $L\ln L(X)$.
 Let $f^*$ be the non-increasing rearrangement of $f$ defined by
 $$f^*(t)=\inf\{s\geq 0:|\{x:|f(x)|>s\}|\leq t\}$$
  and
 \begin{align*}
 \|f\|^*_{L\ln L(X)}=\int^\infty_0 f^*(t)\ln \frac {|X|}{t}dt.
 \end{align*}
 By \cite[Theorem 6.30]{CR16}, we know that $\|\cdot\|^*_{L\ln L}$ is a norm on $L\ln L(X)$. It is easy to see that $\|f\|^*_{L\ln L(X_1)}\leq \|f\|^*_{L\ln L(X)}$ if $X_1\subset X$.

 Let $\Phi$ be a function in the Schwartz space $S(\mathbb{R}^n)$ satisfying $\int_{\mathbb{R}^n}\Phi(x)dx=1$ and set $\Phi_t(x)=\frac{1}{t^n}\Phi(\frac{x}{t})$.
 Following Stein \cite[p. 91]{S93}, we can define the Hardy space $\mathscr{H}^1$ as the space of all tempered distributions $f$ satisfying
 $$\|f\|_{\mathscr{H}^1}=\|\sup\limits_{t>0}|f\ast\Phi_t|\|_{L^1}<\infty.$$

 Let $B_r(x)$ be the geodesic disk in $M$ with centre $x$ and radius $r$. The bounded mean oscillation space $BMO(M)$ consists of all functions $f$ for which
 \begin{align*}
 \|f\|_{BMO}=\sup_{B_r(x)\subset M}\frac{1}{V(B_r(x))}\int_{B_r(x)}|f(y)-\frac{\int_{B_r(x)}f(z)dz}{V(B_r(x))}|dy<\infty.
 \end{align*}

 Let $K_M$ be the Gauss curvature of $M$. Sometimes we assume that there exist $p>1$ and $A>0$ which depend only on $M$ such that
  \begin{align}
 \sup\limits_{x\in M}\|K_M\|_{L^p(B_1(x))}\leq A.\label{sla1.1}
 \end{align}
 For brevity, we have used the same letter $A$ in this expression as in \eqref{sla1.2} to denote a positive constant. Indeed, provided
 $A$ is chosen sufficiently large, we can always replace the constants in both \eqref{sla1.2} and \eqref{sla1.1} with a single universal constant $A$.

 For the sake of convenience, we assume from now on that $u,f,g$ are smooth in the relevant domains. Our main result is the following interior estimate.
 \begin{theorem}\label{main1}
 Suppose that $M$ satisfies (\ref{sla1.2}) and \eqref{sla1.1}. If $B_1$ is a unit geodesic ball in $M$, $f,g\in L\ln L(B_1)$ and $u$ is a solution of (\ref{basic eqn}),
 then
 $$\|u\|_{C(B_{\frac 12})}\leq C(\|u\|_{L^1(B_1)}+\|f\|^*_{L\ln L(B_1)}),$$
 where $C$ depends only on $A,p$ and the structure of $g$.
 \end{theorem}
 \textbf{Remark 1.} One can check that we can replace $f\in L\ln L(B_1)$ by $f\in \mathscr{H}^1_{at}(B_1)$ in this theorem where $\mathscr{H}^1_{at}(B_1)$ is the atomic Hardy space
 defined in \cite{CW2}. If $M=\mathbb{R}^2$, then the atomic Hardy space is the usual Hardy space.

 As special cases of Theorem \ref{main1}, we can obtain the following theorems, which are interesting and new to us.
 \begin{proposition}\label{main2}
 Suppose that the surface $M$ satisfies (\ref{sla1.2}) and \eqref{sla1.1}. If $B_1$ is a unit geodesic ball in $M$, $f\in L\ln L(B_1)$ and
 $$\triangle u(x)=f(x), \quad x\in B_1,$$
 then
$$\|u\|_{C(B_{\frac 12})}\leq C(\|u\|_{L^1(B_1)}+\|f\|^*_{L\ln L(B_1)}),$$
 where $C$ depends only on $A,p$.
 \end{proposition}

 \begin{proposition}\label{main3}
 Let $B_1$ be the unit disk in $\mathbb{R}^2$, $f,g\in L\ln L(B_1)$ and
 $$\triangle u(x)=g(x)u(x)+f(x), \quad x\in B_1.$$
 Then
 $$\|u\|_{C(B_{\frac 12})}\leq C(\|u\|_{L^1(B_1)}+\|f\|^*_{L\ln L(B_1)}),$$
 where $C$ depends only on the structure of $g$.
 \end{proposition}

\textbf{Remark 2.} In particular, if $g\in L^q(B_1)$ for some $q>1$, then we can take the constant $C$ in Proposition \ref{main3} to be $C_q(1+\|g\|_q)^{q_1}$ for any $q_1>\frac{4q}{q-1}$
 which recovers the known result.

 As applications, we can give a simple proof of the following Harnack inequality.
 \begin{corollary}\label{main4}
  Let $B_1$ be the unit disk in $\mathbb{R}^2$. If $f\in L\ln L(B_1)$, $u\geq 0$ and
 $$\triangle u(x)=f(x), \quad x\in B_1,$$
 then we have the following Harnack inequality:
 $$\max\limits_{x\in B_{\frac 12}}u(x)\leq C(\min\limits_{y\in B_{\frac 12}}u(y)+\|f\|^*_{L\ln L(B_1)}).$$
 \end{corollary}

 Since $f\in L\ln L(B_1)$ implies $f-\frac{\int_{B_1}fdx}{\pi}\in \mathscr{H}^1$ (see \cite{S69}), we can extend Corollary~\ref{main4} as follows:
 \begin{corollary}\label{main5}
  Let $B_1$ be the unit disk in $\mathbb{R}^2$. If $f\in \mathscr{H}^1(B_1)$, $u\geq 0$ and
 $$\triangle u(x)=f(x), x\in B_1,$$
 then we have the following Harnack inequality:
 $$\max\limits_{x\in B_{\frac 12}}u(x)\leq C(\min\limits_{y\in B_{\frac 12}}u(y)+\|f\|_{\mathscr{H}^1(B_1)}).$$
 \end{corollary}

 Finally, we prove the following global estimate for Poisson equations on Riemann surfaces.
 \begin{theorem}\label{main6}
 Suppose that $M$ satisfies (\ref{sla1.2}) and \eqref{sla1.1}. If $B_1$ is a unit geodesic ball in $M$, $u\in H^1_0(B_1)$, $f\in L\ln L(B_1)$ and
 $$\triangle u(x)=f(x), x\in B_1,$$
 then
 $$\|u\|_{C(B_1)}\leq C\|f\|^*_{L\ln L(B_1)},$$
 where $C$ depends only on $A,p$.
 \end{theorem}
 \textbf{Remark 3.} Unlike the case discussed in Remark 1, we cannot replace $f\in L\ln L(B_1)$ by $f\in \mathscr{H}^1_{\mathrm{at}}(B_1)$ in the proof of this theorem.

 In Section 2 we derive some results for the surface $M$. We prove Theorem \ref{main1} in Section 3, two Harnack's inequalities in Section 4 and
 Theorem \ref{main6} in Section 5.

 In what follows, the constant $C$ denotes a generic positive constant that may change from line to line,
 but depends only on $A$, $p$, and the structure of $g$.

\section{Some lemmas}

 For any $x_0\in M$, by the semi-geodesic coordinate grid around $x_0$ we can write the first fundamental form as
 $dr^2+G^2(r,\theta)d\theta^2$ where
 $$r>0,0\leq \theta<2\pi,G(r,\theta)\geq 0, G(0,\theta)=0, \partial_rG(0,\theta)=1.$$

 In this note, $G(r,\theta)$ may vary for different $x_0$. But one can see that all the constants are uniformly in $x_0$. So we use the same symbol $G(r,\theta)$ for all points.

 \begin{lemma}\label{low}
 If $M$ satisfies the isoperimetric inequality \eqref{sla1.2}, then for any\\ $x_0\in M,r>0$, there holds
 \begin{align}
 V(B_r(x_0))\geq \frac{r^2}{4A},\quad l(\partial B_r(x_0))\geq \frac{r}{2A}.\label{sla1.3}
 \end{align}
  \end{lemma}
 \begin{proof}
 From (\ref{sla1.2}), we have
 $$\frac{d}{dr}V(B_r(x_0))=l(\partial B_r(x_0))\geq \sqrt{\frac{V(B_r(x_0))}{A}}$$
 which yields that
 $$\frac{d}{dr}\sqrt{V(B_r(x_0))}\geq \frac{1}{\sqrt{4A}}.$$
 So, one obtain that $V(B_r(x_0))\geq \frac{r^2}{4A}$ and $l(\partial B_r(x_0))\geq \frac{r}{2A}$ from \eqref{sla1.2}. This finishes the proof.
 \end{proof}

 \begin{lemma}\label{up}
 If $M$ satisfies \eqref{sla1.1}, then for any $x_0\in M,0<r<1$, there holds
 \begin{align}
 V(B_r(x_0))\leq (2\pi+A)^{p+1}r^2,l(\partial B_r(x_0))\leq (2\pi+A)^{p+1}r.\label{sla1.4}
 \end{align}
 \end{lemma}
 \begin{proof} Note that $K_M(r,\theta)=-\frac{\partial^2_{rr}G(r,\theta)}{G(r,\theta)}$. When $r<1$, by \eqref{sla1.1} we have
  \begin{align}
 l(\partial B_r(x_0))&=\int^{2\pi}_{0}G(r,\theta)d\theta\nonumber\\
 &\leq \int^{2\pi}_{0}\int^{r}_{0}|\partial_{r}G(r,\theta)|dtd\theta\nonumber\\
 &\leq \int^{2\pi}_{0}\int^{r}_{0}\left(1+\int^{t}_{0}|\partial^2_{rr}G(s,\theta)|ds\right)dtd\theta\nonumber\\
 &=2\pi r+\int^{2\pi}_{0}\int^{r}_{0}|\partial^2_{rr}G(s,\theta)|(r-s)dsd\theta\nonumber\\
 &\leq 2\pi r+r\int^{2\pi}_{0}\int^{r}_{0}|K_M(s,\theta)|G(s,\theta)dsd\theta\nonumber\\
 &=2\pi r+r\|K_M\|_{L^1(B_r(x_0))}\nonumber\\
 &\leq 2\pi r+ArV(B_{r}(x_0))^{1-\frac 1p}.\label{sla1.5}
 \end{align}
 Therefore, we can obtain that
 \begin{align*}
 V(B_{r}(x_0))&=\int^{r}_{0}l(\partial B_t(x_0))dt\\
 &\leq\int^{r}_{0}(2\pi t+AtV(B_{t}(x_0))^{1-\frac 1p})dt\\
 &\leq \pi r^2+\frac{Ar^2V(B_{r}(x_0))^{1-\frac 1p}}{2}.
 \end{align*}
 When $V(B_{r}(x_0))\leq 1$, we have
 $$V(B_{r}(x_0))\leq (\pi+\frac{A}{2}) r^2.$$
  When $V(B_{r}(x_0))>1$ and $r<\frac{1}{\sqrt{A}}$, we can get that
 $$V(B_{r}(x_0))\leq \pi r^2+\frac{Ar^2V(B_{r}(x_0))^{1-\frac 1p}}{2}\leq\pi r^2+\frac{V(B_{r}(x_0))}{2}$$
 which implies that
 $$V(B_{r}(x_0))\leq 2\pi r^2.$$
 If $V(B_{r}(x_0))>1$ and $\frac{1}{\sqrt{A}}\leq r<1$, we obtain that
\begin{align*}
 V(B_{r}(x_0))\leq \pi r^2+\frac{Ar^2V(B_{r}(x_0))^{1-\frac 1p}}{2}\leq \pi+\frac{AV(B_{r}(x_0))^{1-\frac 1p}}{2}.
 \end{align*}
 Thus, when $V(B_{r}(x_0))>1$ and $\frac{1}{\sqrt{A}}\leq r<1$, we have
\begin{align*}
 V(B_{r}(x_0))\leq (\pi+\frac A2)^p\leq A(\pi+\frac A2)^pr^2.
 \end{align*}
 So, for any $r<1$, we can always obtain that
 $$V(B_{r}(x_0))\leq (2\pi+A)^{p+1}r^2.$$
 Finally, for any $r<1$, by (\ref{sla1.5}) we have
  \begin{align*}
 l(\partial B_r(x_0))\leq 2\pi r+Ar(2\pi+A)^{(p+1)(1-\frac 1p)}r^{2-\frac 2p}\leq (2\pi+A)^{p+1}r.
 \end{align*}
 This finishes the proof.
\end{proof}

 \begin{lemma}\label{slnL}
  Suppose that $M$ satisfies (\ref{sla1.2}). For any $x_0\in M,R>0$ there holds
 \begin{align*}
 \int_{B_R(x_0)} |f(x)|\int^R_{d(x,x_0)}\frac{1}{l(\partial B_r(x_0))}drdV\leq A\|f\|^*_{L\ln L(B_R(x_0))}.
 \end{align*}
 \end{lemma}
 \begin{proof} Set $h(x)=\int^R_{d(x,x_0)}\frac{1}{l(\partial B_r(x_0))}dr$. From \eqref{sla1.2} and the identity\\
  $\frac{d}{dr}V(B_r(x_0))=l(\partial B_r(x_0))$, we have
\begin{align*}
 h(x)=&\int^R_{d(x,x_0)}\frac{1}{l(\partial B_r(x_0))}dr\\
 =&\int^R_{d(x,x_0)}l(\partial B_r(x_0))^{-2}dV(B_{r}(x_0))\\
 \leq &\int^R_{d(x,x_0)}\frac{A}{V(B_{r}(x_0))}dV(B_{r}(x_0))\\
 = &A\ln \frac{V(B_R(x_0))}{V(B_{d(x,x_0)}(x_0))}.
 \end{align*}
 So we obtain that
 $$h^*(V(B_t(x_0)))\leq A\ln \frac{V(B_R(x_0))}{V(B_t(x_0))}$$
  which yields that
  $$h^*(t)\leq A\ln \frac{V(B_R(x_0))}{t}.$$
 Thus, by \cite[Theorem 4.15]{CR16} we get that
 \begin{align*}
 \int_{B_R(x_0)} |f(x)|\int^R_{d(x,x_0)}\frac{1}{l(\partial B_r(x_0))}drdV
 =&\int_{B_R(x_0)} |f(x)h(x)|dV\\
 \leq &\int^{V(B_R(x_0))}_{0} f^*(t)h^*(t)dt\\
\leq &A\int^{V(B_R(x_0))}_{0} f^*(t)\ln \frac{V(B_R(x_0))}{t}dt\\
=&A\|f\|^*_{L\ln L(B_R(x_0))}.
 \end{align*}
 This finishes the proof.
 \end{proof}

\section{Proof of Theorem \ref{main1}}

 For convenience, we denote the center of $B_1$ by $O$ and $|x|=d(x,O)$ for $x\in B_1$.

 For any $x_0\in B_1,r<1-|x_0|,\theta\in [0,2\pi]$, by the fact $G(0,\theta)=0$, one can get that
 $$|G(r,\theta)-r\partial_r G(r,\theta)|\leq \int^r_0 t|\partial^2_{rr} G(t,\theta)|dt.$$
 Thus, for any $\theta,\varphi\in [0,2\pi]$, we have
  \begin{align*}
 &|\partial_r G(r,\theta)G(r,\varphi)-G(r,\theta)\partial_r G(r,\varphi)|\\
 =&|(G(r,\varphi)-r\partial_r G(r,\varphi))\partial_r G(r,\theta)-(G(r,\theta)-r\partial_r G(r,\theta))\partial_r G(r,\varphi)|\\
 \leq &\int^r_0 t(|\partial^2_{rr} G(t,\varphi)\partial_r G(r,\theta)|+|\partial^2_{rr} G(t,\theta)\partial_r G(r,\varphi)|)dt.
 \end{align*}
 Therefore, for any $0\leq \rho<1-|x_0|$, from  \eqref{sla1.1}, \eqref{sla1.3}, \eqref{sla1.4}, $\partial_r G(0,\theta)=1$, the identity
 $l(\partial B_r(x_0))=\int^{2\pi}_0 G(r,\varphi)d\varphi$ and $K_M(r,\theta)=-\frac{\partial^2_{rr} G(r,\theta)}{G(r,\theta)}$, we obtain that
\begin{align}
 &\int^{\rho}_{0}\int^{2\pi}_{0} |\frac{\partial}{\partial r}(\frac{G(r,\theta)}{l(\partial B_r(x_0))})|d\theta dr\nonumber\\
 =&\int^{\rho}_{0}\int^{2\pi}_{0} \frac{|\partial_r G(r,\theta)l(\partial B_r(x_0))-G(r,\theta)\partial_r l(\partial B_r(x_0))|}{l(\partial B_r(x_0))^2}d\theta dr\nonumber\\
 \leq &C\int^{\rho}_{0}r^{-2}\int^{2\pi}_{0} \int^{2\pi}_0|\partial_r G(r,\theta)G(r,\varphi)-G(r,\theta)\partial_r G(r,\varphi)|d\varphi d\theta dr\nonumber\\
 \leq &C\int^{\rho}_{0}r^{-2}(\int^{2\pi}_{0}|\partial_r G(r,\theta)| d\theta)\left(\int^r_0t\int^{2\pi}_0|\partial^2_{rr} G(t,\varphi)|d\varphi dt\right)dr\nonumber\\
 \leq &C\int^{\rho}_{0}r^{-1}\left(\int^{2\pi}_{0}(1+\int^r_0|\partial^2_{rr} G(t,\theta)|dt) d\theta\right)\left(\int^r_0\int^{2\pi}_0|\partial^2_{rr} G(t,\varphi)|d\varphi dt\right)dr\nonumber\\
 \leq &C\int^{\rho}_{0}r^{-1}(2\pi+\|K_M\|_{L^1(B_r(x_0))})\|K_M\|_{L^1(B_r(x_0))}dr\nonumber\\
 \leq &C\int^{\rho}_{0}r^{-1}(2\pi+V(B_r(x_0))^{1-\frac 1p}\|K_M\|_{L^p(B_1(x_0))})V(B_r(x_0))^{1-\frac 1p}\|K_M\|_{L^p(B_1(x_0))}dr\nonumber\\
 \leq &C\int^{\rho}_{0}r^{-1}(1+r^{2-\frac 2p})r^{2-\frac 2p}dr\nonumber\\
  \leq &C\int^{\rho}_{0}r^{1-\frac 2p}dr=C\rho^{2-\frac 2p}.\label{sl2.1}
\end{align}

 By the divergence theorem, Lemma \ref{slnL}, and \eqref{sl2.1}, for any $\rho<1-|x_0|$, we obtain that
 \begin{align}
 &|u(x_0)-\frac{1}{l(\partial B_{\rho}(x_0))}\int^{2\pi}_{0} u(\rho,\theta)G(\rho,\theta)d\theta|\nonumber\\
 =&\left|\int^{\rho}_{0} \frac{\partial}{\partial r}\left(\frac{1}{l(\partial B_r(x_0))}\int^{2\pi}_{0}u(r,\theta)G(r,\theta)d\theta\right) dr\right|\nonumber\\
 =&\left|\int^{\rho}_{0}\left(\frac{1}{l(\partial B_r(x_0))}\int^{2\pi}_{0} \partial_r u(r,\theta)G(r,\theta)d\theta-\int^{2\pi}_{0} u(r,\theta)\frac{\partial}{\partial r}(\frac{G(r,\theta)}{l(\partial B_r(x_0))})d\theta\right)dr\right|\nonumber\\
 \leq &\left|\int^{\rho}_{0}\left(\frac{1}{l(\partial B_r(x_0))}\int_{B_r(x_0)} \triangle u(x)dV\right)dr\right|+
 \int^{\rho}_{0}\int^{2\pi}_{0} |\frac{\partial}{\partial r}(\frac{G(r,\theta)}{l(\partial B_r(x_0))})|d\theta dr\|u\|_{C(B_{\rho}(x_0))}\nonumber\\
\leq &\int^{\rho}_{0}\frac{1}{l(\partial B_r(x_0))}\int_{B_r(x_0)} (|g(x)u(x)|+|f(x)|)dVdr+
 C\rho^{2-\frac 2p}\|u\|_{C(B_{\rho}(x_0))}\nonumber\\
 =&\int_{B_{\rho}(x_0)} (|g(x)u(x)|+|f(x)|)\int^{\rho}_{|x-x_0|}\frac{1}{l(\partial B_t(x_0))}dtdV+C\rho^{2-\frac 2p}\|u\|_{C(B_{\rho}(x_0))}\nonumber\\
 \leq & C[\|f\|^*_{L\ln L(B_{\rho}(x_0))}+(\|g\|^*_{L\ln L(B_{\rho}(x_0))}+\rho^{2-\frac 2p})\|u\|_{C(B_{\rho}(x_0))}].\label{sl2.3}
 \end{align}
 By the absolute continuity, when $g\in L\ln L(B_1)$, we can find $\rho_0\in(0,\frac 12)$ depending only on $A,p$ and the function $g$ such that
 $$C(\|g\|^*_{L\ln L(B_{\rho_0}(x_0))}+\rho_0^{2-\frac 2p})<\frac 18$$
 for any $x_0\in B_{\frac 12}$. Then from \eqref{sl2.3}, for any $\rho<\rho_0$ we get that
 \begin{align*}
 &|u(x_0)-\frac{1}{l(\partial B_{\rho}(x_0))}\int^{2\pi}_{0} u(\rho,\theta)G(\rho,\theta)d\theta|\\
 \leq &C[\|f\|^*_{L\ln L(B_{\rho}(x_0))}+(\|g\|^*_{L\ln L(B_{\rho}(x_0))}+\rho^{2-\frac 2p})\|u\|_{C(B_{\rho}(x_0))}]\\
 \leq &C\|f\|^*_{L\ln L(B_1)}+\frac {1}{8}\|u\|_{C(B_{\rho}(x_0))}.
 \end{align*}
 Thus, for any $t<\rho_0$, integrating over $(0,t)$ we obtain
  \begin{align}
 &V(B_t(x_0))|u(x_0)|=\int^t_0 l(\partial B_{\rho}(x_0))|u(x_0)|d\rho\nonumber\\
 \leq &\int^t_0\int^{2\pi}_{0} |u(\rho,\theta)|G(\rho,\theta)d\theta d\rho+\int^t_0 l(\partial B_{\rho}(x_0))[C\|f\|^*_{L\ln L(B_1)}+\frac {1}{8}\|u\|_{C(B_{\rho}(x_0))}]d\rho\nonumber\\
  \leq &\|u\|_{L^1(B_t(x_0))}+V(B_t(x_0))(C\|f\|^*_{L\ln L(B_1)}+\frac 18\|u\|_{C(B_t(x_0))}).\label{sl2.4}
 \end{align}
 From \eqref{sla1.3} and \eqref{sl2.4}, we obtain
 \begin{align*}
 |u(x_0)|\leq &\frac{\|u\|_{L^1(B_t(x_0))}}{V(B_t(x_0))}+C\|f\|^*_{L\ln L(B_1)}+\frac 18\|u\|_{C(B_t(x_0))}\\
 \leq &\frac{4A\|u\|_{L^1(B_1)}}{t^2}+C\|f\|^*_{L\ln L(B_1)}+\frac 18\|u\|_{C(B_t(x_0))}.
 \end{align*}

 Clearly, for any $y_0\in B_r(x_0),r,t<\frac{\rho_0}{2}$, we can also get that
 \begin{align}
 |u(y_0)|\leq &\frac{4A\|u\|_{L^1(B_1)}}{t^2}+C\|f\|^*_{L\ln L(B_1)}+\frac 18\|u\|_{C(B_t(y_0))}\nonumber\\
 \leq &\frac{4A\|u\|_{L^1(B_1)}}{t^2}+C\|f\|^*_{L\ln L(B_1)}+\frac 18\|u\|_{C(B_{r+t}(x_0))}.\label{sl2.5}
 \end{align}

 Now we set $\omega(r)=\|u\|_{C(B_r(x_0))}$. Then for any $r,t<\frac{\rho_0}{2}$, \eqref{sl2.5} implies
 \begin{align}
 \omega(r)\leq \frac{4A\|u\|_{L^1(B_1)}}{t^2}+C\|f\|^*_{L\ln L(B_1)}+\frac 18\omega(r+t).\label{sl2.6}
 \end{align}
 Set $t_k=(1-2^{-k})\rho_0$. Then we obtain the iteration
 $$\omega(t_k)\leq \frac{4^{k+2}A\|u\|_{L^1(B_1)}}{\rho_0^2}+C\|f\|^*_{L\ln L(B_1)}+\frac 18 \omega(t_{k+1}).$$
 By induction, we obtain
 $$\omega(t_1)\leq 8^{1-k}\omega(t_k)+\sum\limits^{k-1}_{i=1}8^{1-i}(\frac{4^{i+2}A\|u\|_{L^1(B_1)}}{\rho_0^2}+C\|f\|^*_{L\ln L(B_1)}).$$
 Letting $k\to \infty$, we get that
 \begin{align*}
 \omega(\frac{\rho_0}{2})=&\omega(t_1)
 \leq \lim_{k\to \infty}8^{1-k}\omega(\rho_0)+\sum\limits^{\infty}_{i=1}8^{1-i}(\frac{4^{i+2}A\|u\|_{L^1(B_1)}}{\rho_0^2}+C\|f\|^*_{L\ln L(B_1)})\\
 \leq &C(\frac{\|u\|_{L^1(B_1)}}{\rho_0^2}+\|f\|^*_{L\ln L(B_1)}).
 \end{align*}
 Note that $\rho_0$ depends only on $A,p$ and the function $g$. Due to the arbitrariness of $x_0$, we can yield that
 \begin{align*}
 \|u\|_{C(B_{\frac 12})}\leq\omega(\frac{\rho_0}{2})\leq C(\|u\|_{L^1(B_1)}+\|f\|^*_{L\ln L(B_1)}),
 \end{align*}
 where $C$ depends only on $A,p$ and the function $g$. This finishes the proof.

 \section{Proofs of Harnack inequalities}

 \subsection{Proof of Corollary \ref{main4}}

 For any $x_0\in B_{\frac 12},\rho<\frac 12$, computations similar to \eqref{sl2.3} and Lemma \ref{slnL} yield
  \begin{align*}
  |2\pi u(x_0)-\int_{S^1} u(x_0+\rho\theta)d\theta|\leq &\int_{B_{\rho}(x_0)}|f(x)|\ln \frac{\rho}{|x-x_0|}dx\\
  \leq &C\|f\|^*_{L\ln L(B_{\rho}(x_0))}\leq C\|f\|^*_{L\ln L(B_1)}.
  \end{align*}
 Integrating over $(0,r)$ for $r<\frac 12$, we obtain
 \begin{align*}
 &|\pi r^2u(x_0)-\int_{B_{r}(x_0)} u(x)dx|\\
 \leq &\int^r_0\rho|2\pi u(x_0)-\int_{S^1} u(x_0+\rho\theta)d\theta|d\rho\\
 \leq &C\int^r_0\rho \|f\|^*_{L\ln L(B_1)}d\rho\\
 \leq &Cr^2\|f\|^*_{L\ln L(B_1)}
 \end{align*}
 which implies that
 \begin{align*}
 |u(x_0)-\frac{\int_{B_{r}(x_0)} u(x)dx}{\pi r^2}|\leq C\|f\|^*_{L\ln L(B_1)}.
 \end{align*}
 For any $x_1,x_2\in B_{\frac 12}$ with $d(x_1,x_2)<\frac 14$, it is easy to see that $B_{\frac 14}(x_2)\subset B_{\frac 12}(x_1)$. So we get that
 \begin{align*}
 u(x_2)\leq & \frac{16}{\pi}\int_{B_{\frac 14}(x_2)} u(x)dx+C\|f\|^*_{L\ln L(B_1)}\\
 \leq &\frac{16}{\pi}\int_{B_{\frac 12}(x_1)} u(x)dx+C\|f\|^*_{L\ln L(B_1)}\\
 \leq &4(u(x_2)+C\|f\|^*_{L\ln L(B_1)})+C\|f\|^*_{L\ln L(B_1)}\\
 \leq &C(u(x_2)+\|f\|^*_{L\ln L(B_1)}).
 \end{align*}
 Thus, by iteration, for any $x,y\in B_{\frac 12}$, we have
 \begin{align*}
 u(x)\leq & C(u(y)+\|f\|^*_{L\ln L(B_1)}).
 \end{align*}
 This finishes the proof.

 \subsection{Proof of Corollary \ref{main5}}

 For any $x_0\in B_{\frac 12},\rho<\frac 12$, computations similar to \eqref{sl2.3} yield
 $$|2\pi u(x_0)-\int_{S^1} u(x_0+\rho\theta)d\theta|\leq |\int_{B_{\rho}(x_0)}f(x)\ln \frac{\rho}{|x-x_0|}dx|.$$
 Note that $\ln \frac{\rho}{|x-x_0|}\chi_{B_{\rho}(x_0)}$ is in $BMO$ and its $BMO$-norm is independent of $\rho$. So, by the $\mathscr{H}^1$-$\mathrm{BMO}$ duality, we get that
 $$|2\pi u(x_0)-\int_{S^1} u(x_0+\rho\theta)d\theta|\leq |\int_{B_{\rho}(x_0)}f(x)\ln \frac{\rho}{|x-x_0|}dx|\leq C\|f\|_{\mathscr{H}^1(B_1)}.$$
 Then, using the same proof as for Corollary \ref{main4}, for any $x,y\in B_{\frac 12}$ we get that
 \begin{align*}
 u(x)\leq C(u(y)+\|f\|_{\mathscr{H}^1(B_1)}).
 \end{align*}
 This finishes the proof.

 \section{Global estimate for Poisson equations on $M$}

 We extend $f$ to $B_2$ by setting $f=0$ in $B_2\setminus B_1$ and let $v\in H^1_0(B_2)$ be the solution of $\triangle v=f$ in $B_2$.

 We claim the following John-Nirenberg inequality:
 \begin{align}
 \int_{B_2}e^{\frac{|v(x)|}{2e\sqrt{A}\|\nabla v\|_{L^2(B_2)}}}dx\leq V(B_2).\label{sl5.1}
 \end{align}
 We use a proof similar to that in \cite{GT}. It is well known that $k^k<(2e)^kk!$ for any positive integer $k$. Then, by \eqref{sobolev1.2}, we obtain
  \begin{align*}
 \int_{B_2}e^{\frac{|v(x)|}{2e\sqrt{A}\|\nabla v\|_{L^2(B_2)}}}dx
 =&\int_{B_2}\sum^{\infty}_{k=1}\frac{|v(x)|^k}{(2e\sqrt{A}\|\nabla v\|_{L^2(B_2)})^kk!}dx\\
 =&\sum^{\infty}_{k=1}\frac{\|v\|^k_{L^k(B_2)}}{(2e\sqrt{A}\|\nabla v\|_{L^2(B_2)})^kk!}\\
 \leq &\sum^{\infty}_{k=1}(\frac{k}{2}\sqrt{A})^k\frac{V(B_2)\|\nabla v\|^k_{L^2(B_2)} }{(2e\sqrt{A}\|\nabla v\|_{L^2(B_2)})^kk!}\\
 =&\sum^{\infty}_{k=1}\frac{k^k}{(4e)^kk!}V(B_2)\\
 \leq &\sum^{\infty}_{k=1}2^{-k}V(B_2)=V(B_2).
 \end{align*}
 Thus, we have proved the claim \eqref{sl5.1}.

 Take a constant $C>2e\sqrt{A}$. For any $t>0$, from \eqref{sl5.1} we obtain
  \begin{align*}
 te^{\frac{v^*(t)}{C\|\nabla v\|_{L^2(B_2)}}}\leq &\int^{t}_{0}e^{\frac{v^*(s)}{C\|\nabla v\|_{L^2(B_2)}}}ds\\
 \leq &\int^{V(B_2)}_{0}e^{\frac{v^*(s)}{C\|\nabla v\|_{L^2(B_2)}}}ds\\
 =&\int_{B_2}e^{\frac{|v(x)|}{C\|\nabla v\|_{L^2(B_2)}}}dx\\
 \leq &V(B_2)
 \end{align*}
  which implies that
 \begin{align}
 v^*(t)\leq C\|\nabla v\|_{L^2(B_2)}\ln\frac{V(B_2)}{t}\label{sl5.2}
 \end{align}
 for any $0<t\leq V(B_2)$.

 By \cite[Theorem 4.15]{CR16} and \eqref{sl5.2}, we can obtain that
 \begin{align*}
 \int_{B_2}|v(x)f(x)|dV\leq &\int^{V(B_2)}_0f^*(t)v^*(t)dt\\
 \leq &C\|\nabla v\|_{L^2(B_2)}\int^{V(B_2)}_0f^*(t)\ln\frac{V(B_2)}{t}dt\\
 =&C\|\nabla v\|_{L^2(B_2)}\|f\|^*_{L\ln L(B_1)}.
 \end{align*}

 Now, from the assumptions that $v\in H^1_0(B_2)$ and $\triangle v=f$ in $B_2$, we get that
 \begin{align*}
 \|\nabla v\|^2_{L^2(B_2)}=&|\int_{B_2}v(x)\triangle v(x)dV|\\
 =&|\int_{B_2}v(x)f(x)dV|\\
 \leq &C\|\nabla v\|_{L^2(B_2)}\|f\|^*_{L\ln L(B_1)}
 \end{align*}
 which yields that
  \begin{align}
 \|\nabla v\|_{L^2(B_2)}\leq C\|f\|^*_{L\ln L(B_1)}.\label{sl5.3}
 \end{align}

 On the other hand, by \eqref{sobolev1.2}, Lemma \ref{up} and \eqref{sl5.3}, we get that
 \begin{align}
 \|v\|_{L^1(B_2)}\leq CV(B_2)\|\nabla v\|_{L^2(B_2)}\leq C\|f\|^*_{L\ln L(B_1)}.\label{sl5.4}
 \end{align}

 Though $f\in C^{\infty}(B_1)$, $f$ is not smooth in $B_2$ in general when we extend $f$ to $B_2$ by setting $f=0$ in $B_2\setminus B_1$. Therefore, we cannot apply Theorem 1.1 on $v$ directly. 
 Take $\eta_n\in C^{\infty}_{c}(B_1)$ with $0\leq \eta_n\leq 1, \eta_n\equiv 1$ on $B_{1-\frac 1n}$
 and let $v_n\in H^1_0(B_2)$ be the solution of $\triangle v_n=\eta_n f$ in $B_2$. It is easy to see that $\lim\limits_{n\to \infty}\|f-\eta_n f\|^*_{L\ln L(B_1)}=0$. 
 
 Clearly, all the above computations remain valid if we replace $v,f$ by $v_n,\eta_n f$. \eqref{sl5.4} implies
 $$\lim\limits_{n\to \infty} \|v-v_n\|_{L^1(B_2)}\leq C\lim\limits_{n\to \infty}\|f-\eta_n f\|^*_{L\ln L(B_1)}=0.$$

 As $\triangle (v_n-v_m)=(\eta_n-\eta_m)f$, by Theorem \ref{main1} and \eqref{sl5.4}, we have
 \begin{align*}
 \|v_n-v_m\|_{C(B_{\frac 32})}\leq &C(\|v_n-v_m\|_{L^1(B_2)}+\|(\eta_n-\eta_m)f\|^*_{L\ln L(B_1)})\\
 \leq &C\|(\eta_n-\eta_m)f\|^*_{L\ln L(B_1)}.
 \end{align*}

 Therefore, $\{v_n\}$ is a Cauchy sequence in $C(B_{\frac 32})$. As $v_n\to v$ in $L^1(B_2)$, we have $v_n\to v$ in $C(B_{\frac 32})$.
 Since $u-v$ is harmonic in $B_1$, by the maximum principle, Theorem \ref{main1},
 the assumption $u\in H^1_0(B_1)$ and \eqref{sl5.4} for $v_n$, we can show that
\begin{align*}
\|u\|_{C(B_1)}\leq &\|v\|_{C(B_1)}+\|u-v\|_{C(B_1)}\\
\leq &\|v\|_{C(B_1)}+\|u-v\|_{C(\partial B_1)}\\
\leq &\|v\|_{C(B_1)}+\|v\|_{C(\partial B_1)}\\
\leq &2\|v\|_{C(B_1)}\\
= &2\lim\limits_{n\to \infty}\|v_n\|_{C(B_1)}\\
\leq &C\lim\limits_{n\to \infty}(\|v_n\|_{L^1(B_2)}+\|\eta_nf\|^*_{L\ln L(B_1)})\\
\leq &C\lim\limits_{n\to \infty}\|\eta_nf\|^*_{L\ln L(B_1)}\\
\leq &C\|f\|^*_{L\ln L(B_1)}.
\end{align*}
 This finishes the proof.


\begin{thebibliography}{99}

\bibitem{CR16} Castillo R. E. and Rafeiro C. H., An Introductory Course in Lebesgue Spaces. \textbf{CMS Books in Mathematics.} Springer, [Cham], 2016, xii+461 pp.

	
\bibitem{CW2} Coifman R. R. and Weiss G., Extensions of Hardy spaces and their use in analysis. \textbf{Bull. Amer. Math. Soc.} 83 (1977), no. 4, 569-645.

\bibitem{GT} Gilbarg D. and Trudinger N. S., Elliptic partial differential equations of second order, \textbf{Springer}, 1998.

\bibitem{Helein} H\'elein F., Harmonic maps, conservation laws and moving frames, Second edition, \textbf{Cambridge university press}, 2002.


\bibitem{M1} Moser, J., On Harnack's theorem for elliptic differential equations. \textbf{Commun. Pure Appl. Math.} 14 (1961), 577-591.

\bibitem{S-Y} Schoen R. and Yau S. T., Lectures on differential geometry, \textbf{International Press, Cambridge}, MA, 1994.

\bibitem{Trudinger} Trudinger N. S., On Harnack type inequalities and their application to quasilinear elliptic equations. \textbf{Comm. Pure Appl. Math.} 20 (1967), 721-747.

\bibitem{S69} Stein E. M., Note on the class $L \log L$. \textbf{Studia Math}, 32 (1969), 305-310.

\bibitem{S93} Stein E. M., Harmonic analysis: real-variable methods, orthogonality, and oscillatory integrals. With the assistance of Timothy S. Murphy.
\textbf{Princeton Mathematical Series, 43. Monographs in Harmonic Analysis}, III. Princeton University Press, Princeton, NJ, 1993.







 \end{thebibliography}
\end{document}